\documentclass[12pt]{amsart}

\usepackage{color, graphicx}
\usepackage[usenames,dvipsnames,svgnames,table]{xcolor}
\usepackage{tikz}
\usepackage{url}

\usepackage[a4paper, total={6in, 8.05in}]{geometry}

\usepackage{amsmath}
\usepackage{amsfonts}
\usepackage{latexsym, amssymb}
\usepackage[T1]{fontenc}
\usepackage[english]{babel}
\providecommand{\U}[1]{\protect\rule{.1in}{.1in}}

\newtheorem{theorem}{Theorem}[section]

\newtheorem{conjecture}[theorem]{Conjecture}
\newtheorem{corollary}[theorem]{Corollary}

\newtheorem{definition}[theorem]{Definition}

\newtheorem{notation}[theorem]{Notation}

\newtheorem{proposition}[theorem]{Proposition}
\newtheorem{question}[theorem]{Question}
\newtheorem{remark}[theorem]{Remark}

\newcommand{\N}{\mathbb N}

\newcommand{\LL}{\mathcal{L}}

\DeclareMathOperator{\sat}{sat}

\newcommand{\vs}[1]{\langle #1 \rangle}

\date{}
\begin{document}

\title{The number system in rational base $3/2$ and the $3x+1$ problem}

\author{Shalom Eliahou and Jean-Louis Verger-Gaugry}

\begin{abstract}
The representation of numbers in rational base $p/q$ was introduced in 2008 by Akiyama, Frougny and Sakarovitch, with a special focus on the case $p/q=3/2$. Unnoticed since then, natural questions related to representations in that specific base turn out to intimately involve the Collatz $3x+1$ function. Our purpose in this note is to expose these links and motivate further research into them.
\end{abstract}



\maketitle

\baselineskip0.55cm

\section{Introduction}

The representation of numbers in rational base $p/q$ with $p,q$ coprime positive integers was introduced in 2008 in~\cite{AFS}, with a focus on the case $p/q=3/2$ to study Mahler's $Z$-number problem and the Josephus problem.

\smallskip
Given $n \in \N$, its representation in rational base $3/2$ is the unique expression of the form
$$
n = \sum_{i=0}^k \frac{a_i}2\left(\frac32\right)^i
$$
with digits $a_i \in \{0,1,2\}$ for all $i \ge 0$ and $a_k \neq 0$ if $n \ge 1$. We then denote the corresponding word in $\{0,1,2\}^*$ by
$$
\vs{n}=a_ka_{k-1}\cdots a_0.
$$
We set $\vs{0}=\epsilon$, the empty word, and we identify any word $0w \in \{0,1,2\}^*$ with $w$.
 
\smallskip
Denoting $\LL =\{\vs{n} \mid n \in \N\} \subset \{0,1,2\}^*$, here are two natural questions about $\LL$:
\begin{itemize}
\item[(1)] If $w \in \LL$, which extensions among $w0,w1,w2$ still belong to $\LL$?
\item[(2)] Given $k \ge 1$, what is the lexicographically largest word of length $k$ in $\LL$? Equivalently, what is the largest $n \in \N$ such that $\vs{n}$ is of length $k$?
\end{itemize}
Apparently unnoticed so far, answers to these questions turn out to intimately involve the Collatz $3x+1$ function, namely $T : \N \to \N$ defined by
\begin{equation}\label{eq T}
T(n) = 
\begin{cases} 
n/2 & \textrm{ if } n \equiv 0 \bmod 2, \\
(3n+1)/2 & \textrm{ if } n \equiv 1 \bmod 2.
\end{cases}
\end{equation}
The purpose of this note is to expose these links and motivate further research into them. Its contents are as follows. In Section~\ref{sec basic}, we recall basic properties of the numeration system in rational base $3/2$, in particular the recursive construction of the word $\vs{n}$ in Section~\ref{sec word} and the odometer allowing to construct $\vs{n+1}$ from $\vs{n}$ in Section~\ref{sec odometer}. We also propose a variant of the Collatz conjecture in Section~\ref{sec collatz}. In Sections~\ref{sec extend} and~\ref{sec saturated}, we address the above questions (1) and (2), respectively. Section~\ref{sec questions} concludes the note with two questions and two conjectures closely related to the Collatz conjecture and possibly of the same order of difficulty.

\section{Basics}\label{sec basic}
Following~\cite{AFS}, every $n\in \N$ has a unique expression of the form
$$
n=\frac12\big(a_k(3/2)^k+a_{k-1}(3/2)^{k-1}+\dots+a_0\big)
$$
with $a_i \in \{0,1,2\}$ for all $i$ and $a_k \neq 0$ if $n \ge 1$. Accordingly, we associate to $n$ the word 
$$\langle n \rangle=a_ka_{k-1}\cdots a_0 \in \{0,1,2\}^*,$$
called the \emph{representation of $n$ in rational base $3/2$}. We set $\vs{0}=\epsilon$, the empty word, and we identify any word $0w \in \{0,1,2\}^*$ with $w$. 

\smallskip
By Lemma 4.1 in~\cite{MWT}, the length of the word $\vs{n}$ is given by $|\vs{n}|=\log_{3/2}(n) + O(1)$. 

\begin{definition} We say that a word $w \in \{0,1,2\}^*$ is \emph{admissible} if $w=\vs{n}$ for some $n \in \N$. We denote by $\LL \subset \{0,1,2\}^*$ the language of admissible words, i.e. 
$$\LL=\{\vs{n} \mid n \in \N\}.$$
\end{definition}

As observed in~\cite{AFS}, \emph{the language $\LL$ is prefix-closed}, i.e. stable under taking prefixes. That is, every prefix $w'$ of an admissible word $w=w'w''$ is itself admissible.

\smallskip

Conversely, given an admissible word $w \in \LL$, which extensions among $w0,w1,w2$ remain admissible? The answer involves the Collatz $3x+1$ function. See Proposition~\ref{prop link 1}.

\subsection{Construction of $\vs{n}$}\label{sec word}
Given $n \in \N$, a recursive construction of $\vs{n} \in \{0,1,2\}^*$ may be described as follows. Set $\vs{0}=\epsilon$, the empty word.

\begin{proposition}\label{prop n1} Let $n_0 \in \N \setminus \{0\}$. Denote $\vs{n_0}=a_k\cdots a_0$ with $a_i \in \{0,1,2\}$ for all $i$. Then 
$$
\begin{cases}
\ a_0 \equiv 2n_0 \bmod 3, \\
\ \vs{n_0} = \vs{n_1}a_0 \textrm{ where } n_1= (2n_0-a_0)/3.
\end{cases}
$$
Moreover, we have $a_0=2n_0-3n_1$ and $a_0 \equiv n_1 \bmod 2$.
\end{proposition}
\begin{proof} As $\vs{n_0}=a_k\cdots a_0$ , we have 
$$
\begin{aligned}
n_0 &= \frac12\big(a_k(3/2)^k+\dots+a_1(3/2)+a_0\big) \\
&= \left(\frac12\big(a_k(3/2)^{k-1}+\dots+a_1\big)\right)\frac32+\frac{a_0}2. \\
\end{aligned}
$$
Hence $a_0\equiv 2n_0 \bmod 3$, and setting $n_1=(2n_0-a_0)/3$, we have $\vs{n_1}=a_k\cdots a_1$. The claims follow.
\end{proof}

\subsection{The odometer}\label{sec odometer}

The \emph{odometer} constructs the word $\vs{n+1}$ from the word $\vs{n}$. It is described as a finite automaton in~\cite{AFS}. Here is an equivalent description.

Let $\sigma=(2 \ 1 \ 0)$ be the cyclic permutation $2 \mapsto 1 \mapsto 0 \mapsto 2$. For a word $u \in \{0,1,2\}^*$, let $u^\sigma \in \{0,1,2\}^*$ denote the word obtained by applying $\sigma$ to each letter of $u$. For instance, $12012^\sigma=01201$, then identified to $1201$. 

\begin{proposition}\label{prop odometer} Let $n \in \N$. If $\vs{n}=u0v \in \{0,1,2\}^*$ with $v \in \{1,2\}^*$, then $\vs{n+1}=u2v^\sigma$.
\end{proposition}
\begin{proof} See~\cite[Section 3.2.3, pp. 66-67]{AFS}.
\end{proof}

The rule applies even if $u$ or $v$ is the empty word. In particular, if $\vs{n} = v \in \{1,2\}^*$, then identifying $v$ with $0v$, we get $\vs{n+1}=2v^\sigma$. Table~\ref{tab words} below, listing $\vs{n}$ for $1 \le n \le 27$, was constructed by hand using this description of the odometer. The interested reader will effortlessly extend it much farther.
\begin{table}[h]
$$
\begin{array}{ccccc}
\begin{array}{|r||r|}
\hline
n & \vs{n}\\
\hline
1 & 2 \\
2 & 21 \\
3 & 210 \\
4 & 212 \\
5 & 2101 \\
6 & 2120 \\
7 & 2122 \\
8 & 21011 \\
9 & 21200 \\
\hline
\end{array} & & 
\begin{array}{|r||r|}
\hline
n & \vs{n} \\
\hline
10 & 21202 \\
11 & 21221 \\
12 & 210110 \\
13 & 210112 \\
14 & 212001 \\
15 & 212020 \\
16 & 212022 \\
17 & 212211 \\
18 & 2101100 \\
\hline
\end{array} & &
\begin{array}{|r||r|}
\hline
n & \vs{n}\\
\hline
19 & 2101102 \\
20 & 2101121 \\
21 & 2120010 \\
22 & 2120012 \\
23 & 2120201 \\
24 & 2120220 \\
25 & 2120222 \\
26 & 2122111 \\
27 & 21011000 \\
\hline
\end{array}
\end{array}
$$
\caption{The words $\vs{n}$ for $1 \le n \le 27$}
\label{tab words}
\end{table}

\subsection{The Collatz conjecture}\label{sec collatz}

Dating back to the 1930's, this conjecture deals with iterates $T^{(k)}$ of the Collatz function $T:\N \to \N$ defined in~\eqref{eq T}. 

\begin{conjecture}\label{conj col1} Let $n_0 \in \N$, $n_0 \ge 1$. Then there exists $k \ge 1$ such that $T^{(k)}(n_0)=1$.
\end{conjecture}

Here is an equivalent version, as shown below.
\begin{conjecture}\label{conj col2} Let $n_0 \in \N$, $n_0 \ge 1$. Then the proportion of even numbers in the sequence of iterates $T^{(k)}(n_0)$ tends to $1/2$ as $k$ tends to infinity.
\end{conjecture}

\begin{proposition} Conjectures~\ref{conj col1} and~\ref{conj col2} are equivalent.
\end{proposition}

We thank the anonymous referee for insisting to formally state Proposition 6 and carefully prove it, as opposed to just a remark with informal arguments as in the original version of this article.

\begin{proof}
If Conjecture~\ref{conj col1} holds, then at some point the sequence of iterates $T^{(k)}(n_0)$ collapses to $1,2,1,2,1,2,\dots$, whence Conjecture~\ref{conj col2}. 

Conversely, assume Conjecture~\ref{conj col2} holds. Given an integer $n \ge 1$, let $\Omega(n)=(n,T(n),T^{(2)}(n),\dots)$ denote the trajectory of $n$ under iteration of $T$. Moreover, given $k \ge 1$, let $$\Omega_k(n)=(n,T(n),T^{(2)}(n),\dots,T^{(k-1)}(n)),$$ and let $k_1$ denote the number of odd integers in $\Omega_k(n)$. Assume for a contradiction that Conjecture~\ref{conj col1} fails. Then there is an integer $m \ge 10$ such that $\min \Omega(m)=m$. (Note that $\min \Omega(n)=1$ for $n \le 9$.) We claim that for all $k \ge 1$, we have
\begin{equation}\label{eq upper}
T^{(k)}(m) \le \frac{(3.1)^{k_1}}{2^k} m.
\end{equation}
The proof is straightforward, by induction on $k$. Indeed, \eqref{eq upper} holds for $k=1$, since $m$ is odd by minimality in $\Omega(m)$ and $T(m)=(3m+1)/2=(3+1/m)m/2 \le (3.1)m/2$ since $m \ge 10$. Let now $k \ge 1$ and assume that $T^{(k)}(m) \le (3.1)^{k_1}/2^k m$. Denote $n=T^{(k)}(m)$. If $n$ is even, then $T^{(k+1)}(m)= T(n) =n/2 \le (3.1)^{k_1}/2^{k+1} m$, as desired. If $n$ is odd, then $$T^{(k+1)}(m) = T(n) = (3+1/n)n/2 \le (3+1/m)n/2 \le (3.1)n/2$$  since $n \ge m \ge 10$. Hence $T^{(k+1)}(m) \le (3.1)^{k_1+1}/2^{k+1}m$, as desired. The claim is proved.

Now $(3.1)^{k_1}/2^{k} \ge 1$ if and only if $k_1/k \ge \ln(2)/\ln(3.1)$, and $\ln(2)/\ln(3.1) > 0.6$. By Conjecture 5, for $k$ large enough, the proportion $k_1/k$ will fall below $0.6$. Hence $T^{(k)}(m) \le (3.1)^{k_1}/2^k m < m$, contradicting the minimality of $m$. This completes the proof of the proposition.
\end{proof}

\smallskip
At the time of writing, Conjecture~\ref{conj col1} has been verified up to $n_0 \le 2^{68}$. See~\cite{B}. With this bound, the results of~\cite{E} and its Table 2 on page 54 imply that, besides the trivial cycle $\{1,2\}$ under $T$, no other cycle, if any, can have length less than $114,208,327,604$. See~\cite{L2} for a general reference on the conjecture, and~\cite{T} for recent strong results about it.

\section{Links with the Collatz function}\label{sec links}

We now address the two questions about $\LL=\{\vs{n} \mid n \in \N\}$ asked in the Introduction, and show that in both cases, the answer naturally involves the Collatz function $T$.

\subsection{Extending admissible words}\label{sec extend}
Here is the first question. Given $w \in \LL$, which extensions among $w0,w1,w2$ still belong to $\LL$? 

\begin{proposition}\label{prop link 1} Let $w \in \LL$, say $w = \vs{n}$ for some $n \in \N$. Then \vspace{-0.1cm}
$$
\begin{cases}
\ w1 \in \LL \iff n \equiv 1 \bmod 2, \textrm{ and then } w1=\vs{(3n+1)/2} =\vs{T(n)}. \\
\ w0 \in \LL \iff n \equiv 0 \bmod 2 \iff w2 \in \LL, \textrm{ and then } w0=\vs{3n/2} \textrm{ and }\; w2=\vs{(3n+2)/2}.
\end{cases}
$$
\end{proposition}

\begin{proof} If $\vs{n}1\in \LL$, Proposition~\ref{prop n1} implies $1 \equiv n \bmod 2$ and $w1=\vs{(3n+1)/2}$. Similarly, if $\vs{n}0 \in \LL$ then $0 \equiv n \bmod 2$ and $w0=\vs{3n/2}$. Finally, if $\vs{n}2  \in \LL$ then $2 \equiv n \bmod 2$ and $w2=\vs{(3n+2)/2}$. The stated equivalences follow.
\end{proof}

\begin{corollary}\label{cor n odd} Let $n \in \N$ be odd. Then $\langle T(n) \rangle=\langle n \rangle1$. 
\end{corollary}
\begin{proof} Immediate from the above Proposition.
\end{proof}

\begin{remark}
The map $n \mapsto n/2$ for $n$ even is straightforward in the base $2$ representation of $n$: just remove the last digit, $0$. Symmetrically, Corollary~\ref{cor n odd} shows that the map $n \mapsto (3n+1)/2$ for $n$ odd is straightforward in the base $3/2$ representation of $n$: just add $1$ as rightmost digit. By any chance, could it happen that considering the map $n \mapsto T(n)$ simultaneously in base $2$ and in rational base $3/2$ may lead to new insights on the Collatz conjecture?
\end{remark}

\subsection{Saturated words}\label{sec saturated}

Given a number system $\mathcal{B}$ over a finite set of digits, we say that a positive integer $n \ge 1$ is \emph{saturated} if its representation $\vs{n}_\mathcal{B}$ in base $\mathcal{B}$ is of length strictly less than that of $\vs{m}_\mathcal{B}$ for all $m \ge n+1$. Accordingly, we say that $\vs{n}_\mathcal{B}$ is a \emph{saturated word}. For instance in base $10$, the saturated integers and corresponding words are the $n=10^\ell-1=9\cdots9$. 

\smallskip
The digits $1,2$ of the saturated words in rational base $3/2$ turn out to be unpredictable in some sense and closely related to the $3x+1$ problem. Recall that $|w|$ denotes the length of the finite word $w$.

\begin{notation} Given $k \ge 1$, we denote by $\sat(k)$ the largest positive integer whose representation in base $3/2$ is of length $k$. That is,
$$
\sat(k) = \max \{n \in \N, |\vs{n}| = k\}. 
$$
\end{notation}

\begin{proposition}\label{prop sat} For all $k \ge 1$, $\vs{\sat(k)}$ is the unique admissible word of length $k$ with digits in $\{1,2\}$.
\end{proposition}
\begin{proof} By the odometer as described in Proposition~\ref{prop odometer}, if $\vs{n}$ is of length $k$ and contains a $0$ digit, then $\vs{n+1}$ is of the same length $k$. This shows that the inequality $|\vs{n+1}| > |\vs{n}|$ happens exactly when $\vs{n}$ contains no $0$ digit except possibly as an irrelevant prefix.
\end{proof}

For instance, a look at $\vs{n}$ for $1 \le n \le 27$ in Table~\ref{tab words} yields $\sat(k)$ and the corresponding word for $k \le 7$ in Table~\ref{tab sat} below.
\begin{table}[h]
$$
\begin{array}{|r||c|l|}
\hline
k & \sat(k) & \vs{\sat(k)} \\
\hline
1 & 1 & 2 \\
2 & 2 & 21 \\
3 & 4 & 212 \\
4 & 7 & 2122 \\
5 & 11 & 21221 \\
6 & 17 & 212211 \\
7 & 26 & 2122111 \\
\hline
\end{array}
$$
\caption{The first seven saturated numbers and words}
\label{tab sat}
\end{table}

\noindent
Observe in Table~\ref{tab sat} that $\vs{\sat(k-1)}$ is a prefix of $\vs{\sat(k)}$. Indeed, this is true for all $k \ge 1$.

\begin{proposition}\label{prop sat k-1} For all $k \ge 2$, we have $\vs{\sat(k)}=\vs{\sat(k-1)}a_0$ where $a_0 \in \{1,2\}$ and $a_0 \equiv \sat(k-1) \bmod 2$.
\end{proposition}
\begin{proof} Since $\LL$ is prefix-closed, it follows from Proposition~\ref{prop sat} that the length $k-1$ prefix of $\vs{\sat(k)}$ equals $\vs{\sat(k-1)}$. Thus $\vs{\sat(k)}=\vs{\sat(k-1)}a_0$ where $a_0 \in \{1,2\}$, and $a_0 \equiv \sat(k-1) \bmod 2$ by Proposition~\ref{prop n1}.
\end{proof}

\subsection{The function $U$}\label{sec U} It turns out that the sequence $(\sat(k))_{k \ge 1}$ coincides with the sequence of iterates $U^{(k)}(0)$ of a certain function $U$ closely related to the Collatz function $T$. See Proposition~\ref{prop sat rec} below.

\begin{notation} Let $U:\N \to \N$ be defined as
$$
U(n) = 
\begin{cases}
\ (3n+2)/2=T(n)+n+1 &\textrm{ if } n \equiv 0 \bmod 2,\\
\ (3n+1)/2=T(n) &\textrm{ if } n \equiv 1 \bmod 2.
 \end{cases}
$$
\end{notation}

\begin{remark}\label{rem U and T} We have $U(n) \equiv T(n)+n+1 \bmod 2$ for all $n \in \N$.
\end{remark}

Starting from any $n_0 \in \N$, let us consider the sequence of iterates $U^{(k)}(n_0)$ for $k \ge 1$. The following proposition shows that, for the sequence of corresponding words $\vs{U^{(k)}(n_0)}$ in rational base $3/2$, \emph{each term is a prefix of the next one and the added digit is either $1$ or $2$.}

\begin{proposition}\label{prop u 2}
Let $n \in \N$. Then 
$$
\vs{U(n)}=
\begin{cases}
\vs{n}1 &\textrm{ if } n \equiv 1 \bmod 2, \\
\vs{n}2 &\textrm{ if } n \equiv 0 \bmod 2.
\end{cases}
$$
\end{proposition}
\begin{proof} Straightforward from Proposition~\ref{prop link 1}.
\end{proof}

A particularly interesting sequence of iterates $U^{(k)}(n_0)$ is for $n_0=0$.

\begin{proposition}\label{prop sat rec} For all $k \ge 1$, we have $\sat(k)=U^{(k)}(0)$.
\end{proposition}
\begin{proof} For $k=1$, we have $\sat(1)=2$ and $U(1)=T(1)=2$. For $k \ge 2$, we only need show $\sat(k)=U(\sat(k-1))$. Indeed, it follows from Proposition~\ref{prop sat k-1} that
$$\sat(k)= \frac{3\sat(k-1)+a_0}2$$
where $a_0 \in \{1,2\}$ and $a_0 \equiv \sat(k-1) \bmod 2$. Hence $\sat(k)=U(\sat(k-1))$ by definition of $U$, as desired.
\end{proof}

Consequently, the growth factor of the sequence $\sat(k)$ tends to $3/2$.

\section{Open questions}\label{sec questions}

We conclude this note with two questions and two conjectures.

\begin{question}\label{ques 1} Is the Collatz conjecture equivalent to some behavior of the iteration of the function $U$? For instance, starting with any $n_0 \in \N$, is it true that the proportion of $1$'s in $\vs{U^{(k)}(n_0)}$ tends to $1/2$ as $k$ goes to infinity, as supported by numerical evidence? Note that by Proposition~\ref{prop u 2}, the number of digits ``0'' remains constant along the trajectory of $n_0$.
\end{question}

\begin{table}[h]
\begin{tabular}{ll}
& 2122111221211221212112212112121111122121111111212222211221212212121\\
\noindent
& 1211222212212222221211111221221212211211222211122111222122221211212\\
\noindent
& 122122212122212121211111111212122222111112211121212222122111112112
\end{tabular}
\caption{The word $\vs{\sat(200)}=\vs{U^{200}(0)}$}
\label{tab 200}
\end{table}
More specifically, for $n_0=0$, we have $U^{(k)}(0)=\sat(k)$ for all $k \ge 1$ by Proposition~\ref{prop sat rec}. Having computed $\sat(k)$ up to $k = 2^{20}=1\hspace{0.7mm}048\hspace{0.7mm}576$, we find that the word $w=\vs{\sat(2^{20})}$, of length $k=2^{20}$, counts exactly $525\hspace{0.7mm}067$ digits ``1'' and $523\hspace{0.7mm}509$ digits ``2''. Moreover, for all $1 \le \ell < 17$, the word $w$ contains all $2^\ell$ words in $\{1,2\}^\ell$ as subwords, showing a great degree of randomness of its digits. See Table~\ref{tab 200} for the word $\vs{\sat(200)}$, with exactly 100 digits ``1''. This numerical evidence leads us to the following conjecture.

\begin{conjecture}\label{conj our} The proportion of digits ``1'' in the word $\vs{\sat(k)}$ tends to $1/2$ as $k$ tends to infinity.
\end{conjecture}

Is Conjecture~\ref{conj our} equivalent to the Collatz conjecture? The question arises from the close relationship between the functions $T$ and $U$, Remark~\ref{rem U and T} and the analogy with Conjecture~\ref{conj col2}. At any rate, even if not equivalent to Conjecture~\ref{conj col2}, Conjecture~\ref{conj our} might well be of the same order of difficulty.

\smallskip
In fact, the numerical evidence below Table~\ref{tab 200} leads to an even stronger conjecture. Since $\vs{\sat(k-1)}$ is a prefix of $\vs{\sat(k)}$ for all $k \ge 2$, we may define the limit
$$
\vs{\sat(\infty)} = \lim_{k \to \infty} \vs{\sat(k)},
$$
an infinite word over $\{1,2\}$. Informally, recall that an infinite word $w$ over a finite alphabet $\mathcal{B}$ of cardinality $|\mathcal{B}|=b$ is said to be \emph{normal} if for all $\ell \ge 1$, all $b^\ell$ words in $\mathcal{B}^\ell$ occur as subwords in $w[n]$, the length $n$ prefix of $w$, with frequency tending to $1/b^\ell$ as $n \to \infty$. Our stronger conjecture runs as follows.

\begin{conjecture}\label{conj our 2} The word $\vs{\sat(\infty)}$ is normal.
\end{conjecture}

Our second question is about fully understanding the Collatz map $T$ in rational base $3/2$.

\begin{question}\label{ques n even} For $n$ odd, we have seen that $\vs{T(n)}=\vs{n}1$. For $n$ even, how does one determine $\vs{T(n)}$ from $\vs{n}$? Is any answer susceptible to lead to new insights on the Collatz conjecture?
\end{question}

Note that determining the parity of $n \in \N$ from the digits $a_i$ of the word $\vs{n}=a_k\cdots a_0 \in \LL$ is more cumbersome than in base $2$ or $10$. Indeed, the following criterion is straightforward but not quite practical for calculations by hand.

\begin{proposition} Let $n \in \N \setminus \{0\}$ and $\vs{n}=a_k\cdots a_0 \in \LL$. Then
$$
n \equiv 0 \bmod 2 \ \iff \ \sum_{i=0}^k 2^{k-i}3^ia_i \equiv 0 \bmod 2^{k+2}. 
$$
\end{proposition}
\begin{proof} We have $n=\frac12 (\sum_{i=0}^k (3/2)^ia_i)=2^{-(k+1)}(\sum_{i=0}^k 2^{k-i}3^ia_i)$. Hence $n \equiv 0 \bmod 2 \iff 2^{k+1}n \equiv 0 \bmod 2^{k+2} \iff \sum_{i=0}^k 2^{k-i}3^ia_i \equiv 0 \bmod 2^{k+2}$.
\end{proof}

\smallskip
\noindent
\textbf{Acknowledgment.} We thank the anonymous referee for his/her careful reading and constructive comments.

\bibliography{base32-collatz}
\bibliographystyle{plain}

\vspace{0.3cm}

\noindent
{\small
\textbf{Authors addresses}

\medskip
\noindent
$\bullet$ Shalom {\sc Eliahou}\textsuperscript{a,b},

\noindent
\textsuperscript{a}Univ. Littoral C\^ote d'Opale, UR 2597 - LMPA - Laboratoire de Math\'ematiques Pures et Appliqu\'ees Joseph Liouville, F-62100 Calais, France\\
\textsuperscript{b}CNRS, FR2037, France\\
\email{eliahou@univ-littoral.fr}

\medskip
\noindent
$\bullet$ Jean-Louis {\sc Verger-Gaugry},

\noindent
Universit\'e-Grenoble-Alpes, Universit\'e Savoie Mont Blanc, F-73000 Chamb\'ery, France\\
\email{verger-gaugry.jean-louis@unsechat-math.fr}
}

\end{document}